\documentclass[12pt,reqno,oneside,psamsfonts]{amsart}
\usepackage{inputenc}
\usepackage{amsmath}
\usepackage{amsthm}
\usepackage{amssymb}
\usepackage{amscd}
\usepackage{amsfonts}
\usepackage{amsbsy}
\usepackage{stmaryrd}
\usepackage{subeqnarray}
\usepackage{epsfig}
\usepackage{pdfsync, hyperref,epstopdf}
\usepackage{float}

\theoremstyle{plain}
\newtheorem{Main}{Theorem}

\newtheorem{theorem}{Theorem}[section]
\newtheorem{proposition}[theorem]{Proposition}
\newtheorem{lemma}[theorem]{Lemma}

\newtheorem*{claim*}{Claim}

\theoremstyle{remark}
\newtheorem{remark}[theorem]{Remark}
\newtheorem{definition}[theorem]{Definition}

\numberwithin{equation}{section}

\def\eg{{\em e.g.\ }}

\newfont\bbf{msbm10 at 12pt}
\def\c{{\bf C}}
\def\eps{\varepsilon}

\def\R{{\mathbb R}}

\def\N{{\mathbb N}}
\def\Z{{\mathbb Z}}

\DeclareMathOperator{\leb}{Leb}
\def\lip{\mbox{Lip}}
\def\graph{\mbox{graph}}
\def\loc{\mbox{\tiny loc}}

\def\qand{\quad\mbox{and}\quad}

\def\supp{\mbox{\rm supp}}

\def\diam{\mbox{\rm diam} }

\def\le{\leqslant}
\def\ge{\geqslant}

\newdimen\AAdi%
\newbox\AAbo%
\def\AArm{\fam0 }
\def\AAk#1#2{\setbox\AAbo=\hbox{#2}\AAdi=\wd\AAbo\kern#1\AAdi{}}%
\def\AAr#1#2#3{\setbox\AAbo=\hbox{#2}\AAdi=\ht\AAbo\raise#1\AAdi\hbox{#3}}%

\def\BBone{{\AArm 1\AAk{-.8}{I}I}}%

\newcommand {\CA}{{\mathcal A}}

\newcommand {\CC}{{\mathcal C}}

\newcommand {\CF}{{\mathcal F}}

\newcommand {\CR}{{\mathcal R}}
\newcommand {\CS}{{\mathcal S}}
\newcommand {\CT}{{\mathcal T}}

\newcommand{\disp}{\displaystyle}
\newcommand{\8}{\infty}
\def\ju{J^u}

\def\lu{\lambda^u}
\def\ls{\lambda^s}
\def\logju{\log\ju}
\def\m1{{-1}}

\newcommand{\ninf}{{n\rightarrow\8}}
\newcommand{\ol}{\overline}
\def\S{\Sigma}
\def\s{\sigma}

\def\supp{\mbox{supp}\,}

\newcommand{\wh}{\widehat}

\newcommand{\pardef}{\stackrel{def}{=}}
\newcommand{\ul}[1]{\underline{#1}}
\newcommand{\w}{\mathbf{w}}

\begin{document}
\synctex=1
\floatplacement{figure}{H}
\title{SRB measures for almost Axiom A diffeomorphisms}
\author{Jos\'e F. Alves}
\address{Jos\'e F. Alves\\ Departamento de Matem\'atica, Faculdade de Ci\^encias da Universidade do Porto\\
Rua do Campo Alegre 687, 4169-007 Porto, Portugal}
\email{jfalves@fc.up.pt} 
\urladdr{http://www.fc.up.pt/cmup/jfalves}

\author{Renaud Leplaideur}
\address{Renaud Leplaideur\\ LMBA UMR 6205,
Universit\'e de Brest\\
6 Av. Victor Le Gorgeu, C.S. 93837, 29238 Brest cedex 3,
France}
\email{Renaud.Leplaideur@univ-brest.fr}
\urladdr{http://www.lmba-math.fr/perso/renaud.leplaideur}

\date{\today}

\thanks{JFA was partly supported by Funda\c c\~ao Calouste Gulbenkian, by CMUP, by the European Regional Development Fund through the Programme COMPETE and by  FCT under the projects PTDC/MAT/099493/2008 and PEst-C/MAT/UI0144/2011. RL was partly supported by ANR DynNonHyp and Research in Pair by CIRM}

\subjclass[2010]{37D25, 37D30, 37C40}

\keywords{Almost Axiom A, SRB measure}

\maketitle

\begin{abstract}  We consider a diffeomorphism $f$ of a compact manifold $M$ which is  {\it
Almost Axiom A}, i.e. $f$ is hyperbolic in a neighborhood
of some compact $f$-invariant set, except in some
singular set of neutral points.
We  prove that if there exists some
$f$-invariant set of hyperbolic points with
positive {\it unstable}-Lebesgue measure
such that for every point in this set the stable and unstable leaves are
``long enough'', then $f$
admits  a probability  SRB measure.
\end{abstract}

\tableofcontents

\section{Introduction}
\subsection{Background}
The goal of this paper is to improve results from \cite{leplaideur-aaa}. It addresses the question of SRB measures for non-uniformly hyperbolic dynamical systems. 
We remind that SRB measures are special \emph{physical} measures, which are measures with observable generic points:

For a smooth dynamical system $(M,f)$, meaning that $M$ is a
compact smooth Riemannian manifold  and $f$ is a $C^{1+}$
diffeomorphism acting on $M$,  we recall that the set
$G_\mu$ of generic points for a $f$-invariant ergodic
probability measure on $M$ $\mu$, is the set of points $x$ such
that for every continuous function $\phi$,
\begin{equation}
\lim_\ninf\frac{1}{n}\sum_{k=0}^{n-1}\phi\circ f^k(x)=\int
\phi\,d\mu. \label{equ1}
\end{equation}
This set has full $\mu$ measure. Though holding for a big set of points in terms of $\mu$ measure, this convergence can be actually observed only when this set $G_\mu$ has strictly positive measure with respect to the volume of the manifold, that is with respect to the Lebesgue measure on $M$. This volume measure will be denoted by $\leb_M$. A measure $\mu$  is said to be  \emph{physical} if $\leb_{M}(G_{\mu})>0$.

The conditions yielding the existence of physical measures for non-uniformly hyperbolic systems has been studied a lot since the 90's (see \eg \cite{young-benedicks,Hu,Viana-Bonatti,alves-bonatti-viana}) but still remains not entirely solved. 
The main reason for that is that physical measure are usually produced under the form of Sinai-Ruelle-Bowen (SRB) measures and 
there is no general theory to construct the particular Gibbs states that are these SRB measures.  One of the explanation of the lack of general theory is probably the large number of ways 
to degenerate the uniform hyperbolicity. 

In   \cite{leplaideur-aaa}, the author studied a way where the loss of hyperbolicity was due to a critical set $S$ where there was no expansion and no contraction even if the general ``hyperbolic'' splitting  with good angles was still existing on this critical set. 
Moreover, the non-uniform hyperbolic hypotheses was inspired by the definition of Axiom-A (where there is forward contraction in the stable direction and backward contraction in the unstable direction) but asked for contraction with a $\limsup$. The main goal of this paper is to prove that this assumption can be released and replaced by a $\liminf$. Such a result is optimal because this is the weakest possible assumption in the ``hyperbolic world''.

We also emphasize a noticeable  second improvement: we prove that the SRB measure constructed is actually a probability measure. 


As this paper is an improvement of \cite{leplaideur-aaa}, the present paper has a very similar structure to \cite{leplaideur-aaa} and the statements are also similar.

\subsection{Statement of results}
We start by defining the class of dynamical systems that we are going to consider.

\begin{definition}\label{def-aaa}
Given  $f:M\to M$ a $C^{1+}$ diffeomorphism and $\Omega\subset M$ a compact $f$-invariant set, we say that $f$ is \emph{Almost Axiom A on $\Omega$}
 if there exists an open set $U\supset\Omega$ such that:
 \begin{enumerate}
   \item for every  $x\in U$ there is a $df$-invariant
splitting (invariant where it makes sense) of the tangent space
$T_xM=E^u(x)\oplus E^s(x)$ with
$x\mapsto E^u(x)$ and $x\mapsto E^s(x)$ H\"older continuous
(with uniformly bounded H\"older constant);
   \item there exist
continuous nonnegative real functions
$x\mapsto k^u(x)$ and $x\mapsto k^s(x)$ such that (for some choice of a
Riemannian norm $\|\ \|$ on $M$)
 for all $x\in U$
 \begin{enumerate}
   \item $\|df(x)v\|_{f(x)}\leq e^{-k^s(x)}\|v\|_x,\quad \forall v\in E^s(x)$,
   \item $\|df(x)v\|_{f(x)}\geq e^{k^u(x)}\|v\|_x,\quad \forall v\in E^u(x)$;
 \end{enumerate}
\item the \emph{exceptional set}, $\disp S=\{x\in U,\ k^u(x)=k^s(x)=0\}$, satisfies $f(S)=S$.
 \end{enumerate}
\end{definition}


From here on we assume that $f$ is Almost Axiom A on $\Omega$. The sets $S$ and $U\supset \Omega$ and the functions $k^u$ and $k^s$ are fixed as in the definition, and the splitting $T_xM=E^u(x)\oplus E^s(x)$ is called the \emph{hyperbolic splitting}.



\begin{definition}\label{def-point-integ}
A point $x\in \Omega$ is called a \emph{point of integration} of the hyperbolic splitting if
there exist $\eps>0$ and  $C^1$-disks $D^u_\eps(x)$ and
$D^s_\eps(x)$ of size $\eps$ centered at $x$ such that
 $T_yD^i_\eps(x)=E^i(y)$ for all $y\in D^i_\eps(x)$
and $i=u,s$.
\end{definition}

By definition, the set of points of integration is invariant by $f$. As usual, having the two families of local stable and unstable manifolds defined, we define
$$\CF^u(x)=\bigcup_{n\geq 0}f^nD^u_{\eps(-n)}(f^{-n}(x))\quad \text{and}\quad
\CF^s(x)=\bigcup_{n\geq 0}f^{-n}D^s_{\eps(n)}(f^{n}(x)),$$ where
$\eps(n)$ is the size of the disks associated to $f^n(x)$. They are called the global stable and unstable manifolds, respectively.


Given a point of integration of the hyperbolic splitting $x$, the manifolds
$\CF^u(x)$ and $\CF^s(x)$ are also immersed Riemannian manifolds. We
denote by $d^u$ and $d^s$ the Riemannian metrics, and by
$\leb^u_x$ and $\leb^s_x$ the Riemannian measures, respectively in $\CF^u(x)$ and $\CF^s(x)$. If a measurable partition is subordinated to the unstable foliation $\CF^u$
(see \cite{rohlin1} and \cite{ledrap1}), any $f$-invariant measure
admits a unique system of conditional measures with respect to the
given partition.
\begin{definition}
We say that an invariant  and ergodic probability
having absolutely continuous conditional measures on unstable
leaves $\CF^u(x)$ with respect to $\leb^u_x$ is a
\emph{Sinai-Ruelle-Bowen (SRB) measure}.
\end{definition}

\begin{definition} \label{def.reg}
Given $\lambda>0$, a point $x\in \Omega$ is said to be
\emph{$\lambda$-hyperbolic}   if 
$$\liminf_{n\to+\8} \frac1{n}\log\|df^{-n}(x)_{|E^{u}(x)}\|\le-\lambda\qand\liminf_{n\to+\8} \frac1{n}\log\|df^{n}(x)_{|E^{s}(x)}\|\le-\lambda.$$
\end{definition}

\begin{definition}
Given $\lambda>0$ and $\eps_0>0$, a point $x\in \Omega$ is called \emph{$(\eps_0,\lambda)$-regular} if the following conditions hold:
\begin{enumerate}
  \item $x$ is  $\lambda$-hyperbolic and a point of integration of the hyperbolic splitting;
  \item $\CF^i(x)$ contains a disk $D^i_{\eps_0}(x)$
 of size
$\eps_0$ centered at $x$, for $i=u,s$.
\end{enumerate}
An $f$-invariant compact set  $\Lambda\subset \Omega$ is said to be  an
\emph{$(\eps_0,\lambda)$-regular set} if all points in $\Lambda$ are
$(\eps_0,\lambda)$-regular.
\end{definition}

\begin{Main}\label{Theorem A} Let  $\Lambda$ be an
$(\eps_0,\lambda)$-regular set.  If there exists some point  $x_0\in \Lambda$ such
that
 $\leb_{D^u_{\eps_0}(x_0)}(D^u_{\eps_0}(x_0)\cap\Lambda)>0$,
then $f$ has a SRB measure.
\end{Main}

Note that the hypothesis of Theorem~\ref{Theorem A} is very
weak. If there exists some probability SRB measure, $\mu$, then,
there exists some $(\eps_0,\lambda)$-regular set
$\Lambda$ of full $\mu$ measure such that for $\mu$-a.e. $x$ in
$\Lambda$, $\leb^u_{x}(D^u_{\eps_0}(x)\cap\Lambda)>0$. On the other hand, a
work due to M. Herman (\cite{Herman}) proves that there exist some
dynamical systems on the circle such that Lebesgue-almost-every point
is ``hyperbolic'' but there is no  SRB measure (even
$\sigma$-finite).

The existence of stable and unstable leaves is crucial to define the  SRB measures. These existences are equivalent to the existence of integration points for a  hyperbolic splitting. This is well known when the hyperbolic splitting is dominated, because it yields the existence of an uniform spectral gap for the derivative. A very close result is also well known in the Pesin Theory, but only on a set of full measure, and then, the precise topological characterization of the set of points of integration given by Pesin Theory depends on the choice of the invariant measure.

In our case the splitting is not dominated and we have no given invariant measure to use Pesin theory. Nevertheless, using the graph transform, we prove integrability even in the presence of indifferent points for a set of points whose precise characterization does not depend on the ergodic properties of some invariant probability measure which would be given {\it a priori}.

\begin{definition}  
A  $\lambda$-hyperbolic point $x\in \Omega$ is said to be of \emph{bounded type} if 
there exist two increasing sequences of integers $(s_{k})$ and $(t_{k})$ with
$$ \limsup_{k\to+\8} \frac{s_{k+1}}{s_{k}}<+\8\qand \disp \limsup_{k\to+\8} \frac{t_{k+1}}{t_{k}}<+\8$$
such that
$$\disp\lim_{k\to+\8} \frac1{s_{k}}\log\|df^{-s_{k}}(x)_{|E^{u}(x)}\|\le-\lambda\qand\lim_{k\to+\8} \frac1{t_{k}}\log\|df^{t_{k}}(x)_{|E^{s}(x)}\|\le-\lambda.$$
\end{definition}
Observe that the last requirement on these sequences is an immediate consequence  Definition~\ref{def.reg}. We  emphasize that the assumption of being of bounded type is weak. Unless the sequence $(s_{k})$ is factorial it is of bounded type. An exponential sequence for instance is of bounded type. Thus, this assumption does not yield that the sequence $(s_{k})$ has positive density. 

\begin{Main}\label{Theorem B} Every $\lambda$-hyperbolic
point  of bounded type is a point of integration of the hyperbolic splitting.
\end{Main}

\subsection{Overview} The rest of this paper proceeds as follows: in Section \ref{sec-rec}
we construct three different generations of rectangles satisfying some Markov property. The last generation is a covering of a ``good zone'' sufficiently far away from the critical zone, and points there have good hyperbolic behaviors. 

In Section \ref{sec-proofthA}, we prove Theorem A. Namely we construct a measure for the return into the cover of rectangles of third generation, and then we prove that this measure can be \emph{opened out } to a $f$-invariant  SRB measure.


In Section \ref{sec-proofthB} we prove Theorem B, that is the existence of stable and unstable manifolds.

The proofs are all based on the same key point: the estimates are all uniformly
hyperbolic outside some fixed bad
neighborhood $B(S,\eps_1)$  of $S$. We show that a $\lambda$-hyperbolic
point cannot stay too long in this fixed neighborhood (see e.g. Lemmas \ref{lem-hyperliminfmedida} and \ref{lem-prA-n+-fini}). Then, an incursion in $B(S,\eps_1)$
cannot spoil too much the (uniformly) hyperbolic estimates of
contractions or expansions.

Obviously, all the constants appearing  are strongly correlated, and special care is taken in Section~\ref{sec-rec} in choosing them in the right
order.



\section{Markov rectangles}\label{sec-rec}

\subsection{Neighborhood of critical zone}

Let $\Lambda$ be some fixed $(\eps_0,\lambda)$-regular set
satisfying the hypothesis of Theorem~\ref{Theorem A}. The
goal of this section is to construct Markov rectangles covering a large part of $\Lambda$ and then to use Young's method; see e.g. \cite{Young-Hu}.
This type of construction has already been implemented  in \cite{leplaideur-aaa}. We adapt it here and focus on the steps where the new assumption (namely $\liminf$ instead of $\limsup$) produces some changes. 

To control the lack of hyperbolicity close to the critical set $S$ we will introduce several constants. Some are directly related to the map $f$, some are related to the set $\Lambda$ and others depend on previous ones. Their dependence will be established in Subsection~\ref{subset-rect3}.

Given
$\eps_1>0$ we define $$B(S,\eps_1)= \{y\in M, d(S,y)<\eps_1\},$$
and
 $$\Omega_0= \Omega\setminus
B(S,\eps_1),\quad
\Omega_1=\Omega_0\cap f(\Omega_0)\cap f^{-1}(\Omega_0)\quad\text{and}
\quad\Omega_2=\Omega_0\setminus\Omega_1.$$

\begin{figure}
\begin{center}
\includegraphics[scale=0.55]{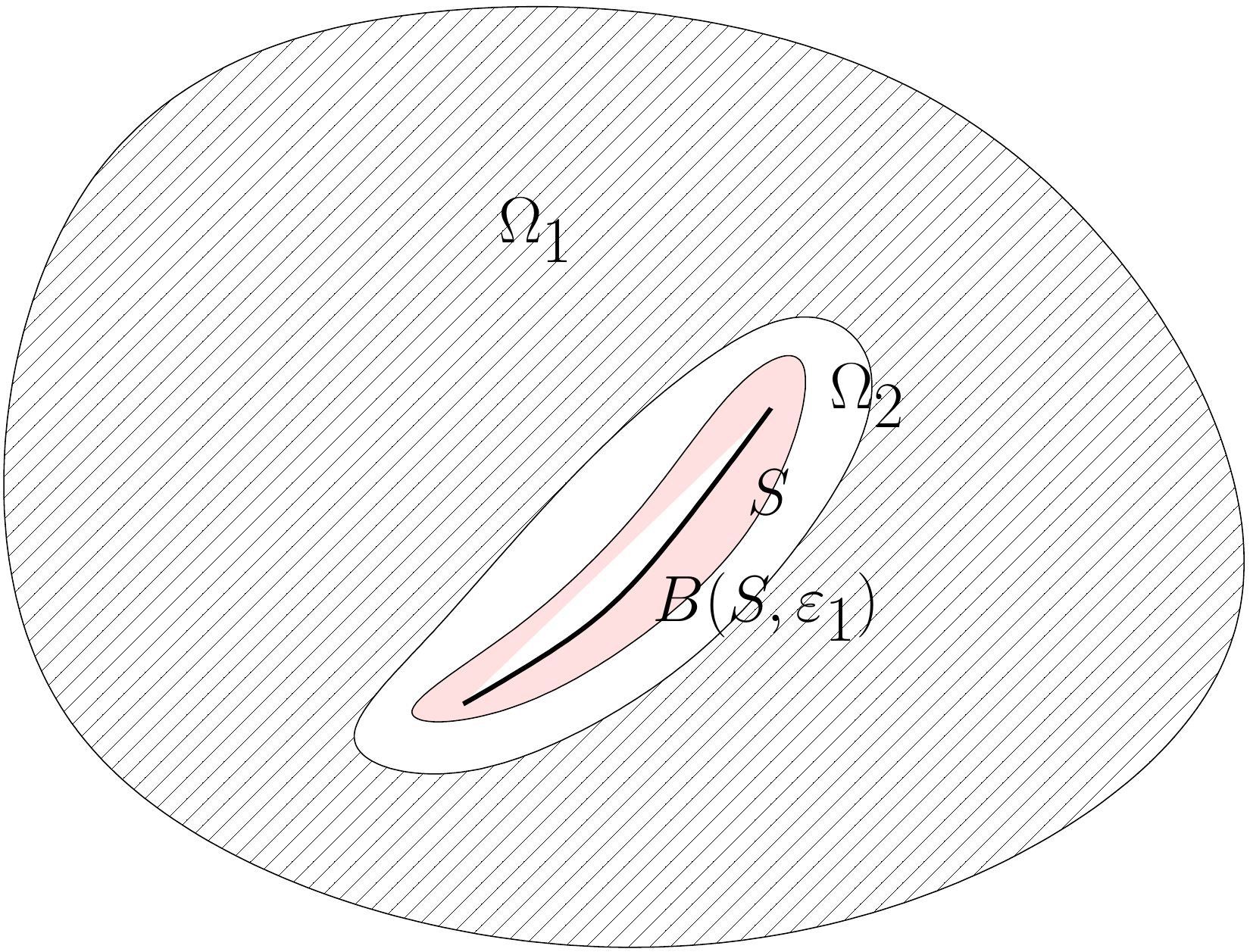}
\caption{The sets $\Omega_{i}$}
\label{fig-omegai}
\end{center}
\end{figure}

The main idea is to consider a new dynamical system, which is in spirit equal to the first return to $\Omega_{0}$. This system will be obtained as the projection of a subshift with countable alphabet, this countable shift being obtained via a shadowing lemma.  This will define rectangles of first generation. To deal with these countably many rectangles we will need to do some extra-work to be able to ``cut''  them. 
This will be done by defining new rectangles qualified of second and third generation.

We assume that
$\eps_1>0$ has been chosen sufficiently  small so that for every
$x$ in $\Omega_2$,
 $$
1\leq\|df(x)_{|E^u(x)}\|<e^{2\frac\lambda3}\qand\disp
1\leq\|df^{-1}(x)_{|E^s(x)}\|<e^{2\frac\lambda3}.$$
Then we fix $0<\eps_2<1$  and define  $\Omega_3=\Omega_3(\eps_2)$ as the set of
points
$x\in\Omega$ such that
\begin{equation}
\label{equ-omega3}
 \min\left( \log \|df_{|E^u(x)}\|,\, -\log
\|df^{-1}_{|E^u(x)}\|,\, -\log \|df_{E^s(x)}\|, \,\log
\|df^{-1}_{|E^s(x)}\|\right) \geq
\eps_2\lambda.
\end{equation}

\begin{definition}\label{def-n+}
Let $x$ in $\Omega_0$.
If $f(x)\in B(S,\eps_1)$, we define the \emph{forward length of stay} of $x$ in $B(S,\eps_1)$ as
$$\disp n^+(x)=\sup\left\{n \in \N:\,
 f^k(x)\in B(S,\eps_1),\quad \text{for all } 0<k<n\right\}.$$
If $f^{-1}(x)\in B(S,\eps_1)$, we define the \emph{backward length of stay} of $x$ in $B(S,\eps_1)$ as
$$\disp n^-(x)=\sup\left\{n \in \N:\,
 f^{-k}(x)\in B(S,\eps_1),\quad \text{for all } 0<k<n\right\}.$$
 Observe that any of these numbers may be equal to $+\infty$.
\end{definition}

\subsection{First generation of rectangles}
A first class of rectangles is constructed by adapting the classical Shadowing Lemma 
 to our case. 
 As mentioned above,  this was done in~\cite[Section 3]{leplaideur-aaa}; see Proposition 3.4. It  just needs local properties of $f$ and the dominated splitting $E^{u}\oplus E^{s}$, not requiring any other hyperbolic properties than the local stable and unstable leaves. Therefore, the new assumption  with $\liminf$ instead of $\limsup$ in the definition of $\lambda$-hyperbolic point does not produce any change  in that construction.  We summarize here the essential properties for these rectangles. 
 
 There is a subshift  of finite type $(\S_{0},\sigma)$  with a countable alphabet $\CA:=\{a_{0},a_{1},\ldots\}$
and a map $\Theta:\S_{0}\to M$ such that the following holds: 

\begin{enumerate}
\item $\Theta(\S_{0})$ contains $\Lambda\cap \Omega_{0}$;
\item the dynamical system $(\S_{0},\sigma)$ induces a dynamical system $(\Theta(\S_{0}),F)$ commuting the diagram
 $$
\begin{array}{rcl}
 \Sigma_0&\stackrel{\sigma}{\longrightarrow}&\Sigma_0\\
\Theta\downarrow&&\downarrow\Theta\\
\Theta(\S_{0})&\stackrel{F}{\longrightarrow}&\Theta(\S_{0});
\end{array}
$$
\item for each  $x\in\Omega_{0}$ we have: $F(x)=f(x)$ if $x\in \Omega_{1}$; $F(x)=f^{n^{+}(x)}(x)$ if $x\in \Omega_{2}\cap f^{-1}(B(S,\eps_{1}))$; and $F^{-1}(x)=f^{-n^{-}(x)}(x)$ if $x\in \Omega_{2}\cap f(B(S,\eps_{1}))$. 
\end{enumerate}

Roughly speaking,  $F$ is the first return map into $\Omega_{0}$ (actually it is defined on the larger set $\Theta(\S_{0})$) and $\s$ is a symbolic representation of this dynamics. 

We define the \emph{rectangles of first generation} as the sets $T_{a_{n}}=\Theta([a_{n}])$, where $[a_{n}]$ is the cylinder in $\S_{0}$ associated to the symbol $a_{n}$, i.e. $[a_{n}]$ is the set of all sequences $\ul x=(x({k}))_{k\in\Z}$ in $\S_{0}$ with $x({0})=a_{n}$.

We remind that $[\ul y,\ul z]$ denotes in $\S_{0}$ the sequence with the same past than $\ul y$ and the same future than $\ul z$. And $[y,z]$ denotes the intersection of  $W^{u}_{loc}(y)$ and $W^{s}_{loc}(z)$. Rectangle means that for $y$ and $z$ in $T_{i}$, $[y,z]$ is also in $T_{i}$.

We may take the rectangles  of first generation with size as small as want, say $\delta>0$, and the following additional properties:
\begin{enumerate}
\item they respect the local product structures in $M$ (at least for the points in $\Lambda$) and in $\S_{0}$: if  we set $W^{u,s}(x,T_{a_{n}})=D^{u,s}_{2\rho}(x)\cap T_{a_{n}}$ for each $x\in T_{a_{n}}$, then  for all $\ul y,\ul z\in [a_{n}]$, 
$$\Theta([\ul y,\ul z])=[\Theta(\ul y),\Theta(\ul z)].$$

\item they are \emph{almost Markov}:
if $x=\Theta(\ul x)$ with $\ul x=x(0)x(1)x(2)\ldots$ and $x(i)$ in the alphabet $\CA$, then for every $j\ge 0$ we have
$$F^{j}(W^{u}(x,T_{x(0)}))\subset W^{u}(F^{j}(x),T_{x(j)})$$
and 
$$
F^{j}(W^{s}(x,T_{x(0)}))\subset W^{s}(F^{j}(x),T_{x(j)}).$$
\end{enumerate}

This last item does not give a full Markov property because it depends on the existence of some code $\ul x$ connecting the two rectangles $T_{x(0)}\ni x$ and $T_{x(j)}\ni F^{j}(x)$.

\subsection{Second generation rectangles}
To get a full Markov property we need to cut the rectangles of first generation as in \cite{bowen}. Due to the non-uniformly hyperbolic settings, we first need to reduce the rectangles to a set of good points with good hyperbolic properties. 

Roughly speaking, the second generation of rectangles we construct here is based  on the following process: we select points in the rectangles $T_{j}$ which are 
\begin{itemize}
\item $\lambda$-hyperbolic,
\item $\leb^{u}$ density points with this property,
\end{itemize}
We show  in Lemma \ref{lem-hyperliminfmedida} that these points satisfy some good property related to the use of the Pliss lemma. 
Then, this set of points is ``saturated'' with respect to the local product structure.

The goal of the next lemma is to use Pliss Lemma. A similar version was already stated in \cite{leplaideur-aaa} but here, exchanging the assumption $\limsup$ by $\liminf$ plays a role.

\begin{lemma}
\label{lem-hyperliminfmedida}
There exists $\zeta>0$ such that for every $\lambda$-hyperbolic point $x$
\begin{equation}
\liminf_\ninf
\frac1n\sum_{j=0}^{n-1}\log\|df^{-1}_{|E^u(f^j(x))}\|<
-\zeta,\label{hypoliminf-u-theoB}
\end{equation}
\end{lemma}

\begin{proof}

Recall that for $0<\eps_2<1$ we have defined $\Omega_3$ as the set of
points
$x\in\Omega$ such that
$$ \min( \log \|df_{|E^u(x)}\|, -\log
\|df^{-1}_{|E^u(x)}\|, -\log \|df_{E^s(x)}\|, \log
\|df^{-1}_{|E^s(x)}\|) \geq
\eps_2\lambda.$$
Let $x$ be a $\lambda$-hyperbolic point. 
We set
 $$
 \delta_{\eps_2}= \disp \liminf\frac1n\#\{0\leq k<n,\
f^k(x)\in \Omega_3\},
$$ 
where $\#A$ stands for the number of elements in a finite set $A$. As $x$ is $\lambda$-hyperbolic, there exists infinitely many values of $n$  such that 
$$\log \| df^{n}_{|E^{s}(x)}\|\le -\frac\lambda2 n.$$
For such an $n$, $\#\{0\leq k<n,\
f^k(x)\in \Omega_3\}$ must be big enough to get the contraction. To be more precise
\begin{equation}\label{equ-preu-delta-2'}
\delta_{\eps_2}\geq
\frac{1-\eps_2}{\frac{\log\kappa}{\lambda}-\eps_2}>0.
\end{equation}
This implies that
$$\liminf_\ninf
\frac1n\sum_{j=0}^{n-1}\log\|df^{-1}_{|E^u(f^j(x))}\|<
-\delta_{\eps_2} \eps_2\lambda.$$
\end{proof}

\begin{remark}
\label{rem-Ws-boncontract}
We emphasize that if \eqref{hypoliminf-u-theoB} holds for $x$, then it holds for every $y\in\CF^{s}(x)$. 
\end{remark}

The main idea of this part is to restrict the $F$-invariant set
$\bigcup_{i\in
\N}T_{a_{i}}$ ($=\Theta(\Sigma_0)$) to some subset satisfying some good
properties. Using Lemma~\ref{lem-hyperliminfmedida} we arrive to the same conclusion of  \cite[Proposition~3.8]{leplaideur-aaa}].

\begin{proposition}\label{prop-omega3}
There exists some $\lambda$-hyperbolic set $\Delta\subset
\Lambda$ such that
\begin{enumerate}
\item $\leb^u(\Delta)>0$;
\item every $x\in\Delta$ is a density point of $\Delta$
for $\leb^u_x$;
\item there exist some $\zeta>0$ such that for every $x\in\Delta$
\begin{equation}
\liminf_\ninf
\frac1n\sum_{j=0}^{n-1}\log\|df^{-1}_{|E^u(f^j(x))})\|<
-\zeta.\label{hypo-u-theoB}
\end{equation}

\end{enumerate}
 \end{proposition}
\begin{proof}
Lemma \ref{lem-hyperliminfmedida}   gives a value $\zeta$  which works for every point in $\Lambda$. 
Defining 
$$\Delta_0=
\bigcup_{i\in\N}\{x\in T_{a_{i}}: \exists (y,z)\in
\Lambda^2\ \text{ such that } x\in \CF^u(y)\cap \CF^s(z)\}.$$
we have that $\Delta_0$ is an $F$-invariant set of $\lambda$-hyperbolic
points in $\bigcup_{i\in\N} T_{a_{i}}$ such that \eqref{hypo-u-theoB} holds for every $x\in \Delta_0$ (
remind Remark \ref{rem-Ws-boncontract}). Moreover, $\leb^u(\Delta_0)>0$.

Now, we recall  that the stable holonomy is
absolutely continuous with respect to Lebesgue measure on the unstable manifolds.
Therefore, if $x$ is a density point of $\Delta_0$ for $\leb^u_x$, every $y$ in $\CF^s(x)\cap \Delta_0$ is also a density
point of $\Delta_0$ for $\leb^u_y$. 
Let $\disp \Delta$ be the set of density points in
$\Delta_0$ for $\leb^u$. This set is $F$-invariant and stable by
intersections of stable and unstable leaves. Hence, every point 
$y\in \Delta$ is also a density point of $\Delta$ for $\leb^u_y$.
\end{proof}

For each $i\in \N$, let $S_{a_{i}}$ be the restriction of the
rectangle $T_{a_{i}}$ to $\Delta$, i.e. $$S_{a_{i}}=T_{a_{i}}\cap\Delta.$$
Given $x\in S_{a_{i}}$ we set 
$$W^u(x,S_{a_{i}})= W^u(x,T_{a_{i}})\cap \Delta\quad\text{and}\quad W^s(x,S_{a_{i}})= W^s(x,T_{a_{i}})\cap \Delta.$$ By construction of
$\Delta$, if $x$ and $y$ are in $S_{a_{i}}$, then $[x,y]$ is also in
$S_{a_{i}}$. As before we have
$$\{[x,y]\}=W^s(x,S_{a_{i}})\cap W^u(y,S_{a_{i}})=D^s_{2\beta(\delta)}(x)\cap
D^u_{2\beta(\delta)}(y).$$
These sets $S_{a_{i}}$ are called \emph{rectangles of second generation} and
their collection is denoted by $\CS$.

\subsection{Third generation of rectangles}\label{subset-rect3}

From the beginning we have introduced several constants. It is
time now to fix some of them. For that purpose we  want first to
summarize how do these constants depend on one another.

\subsubsection{The constants}The set $\Lambda$ is fixed and so the
two constants
$\eps_0$ and
$\lambda$ are also fixed. Note that $\eps_{0}$ imposes a constraint to define the local product structure for points in $\Lambda$. 

We fix $\zeta$ as in \eqref{hypo-u-theoB} and pick  $\eps>0$ small compared to $\lambda$ and $\zeta$, namely $\eps<\zeta/10$.

We  choose $\eps_2>0$ sufficiently small such that every $x\in\Delta$ with $\disp
\log\|df^{-1}_{|E^u(x)}\|<-{\zeta}/{3}$ necessarily belongs to $\Omega_3$.
This is possible because  the exponential  expansion or contraction in the unstable and stable directions $k^u$ and $k^s$ vanishe at the same time.

We take $\eps_1>0$ sufficiently small such that  $\Omega_3$ is
a closed subset of
$\Omega$ not intersecting $\overline{\Omega_2\cup
B(S,\eps_1)}$. Therefore, there exists some $\rho>0$ such that
$d(\Omega_3,\overline{\Omega_2\cup B(S,\eps_1)} )>\rho$. 
Then, $\eps_1$ being fixed,  we choose
the size $\delta>0$ for rectangles of first generation as small as wanted. 
In particular, we assume that 
$2\delta<{\rho}/{10}.$

Each symbol $a_{i}$ is associated to a point $\xi_{a_{i}}\in \Lambda\cap \Omega_{0}$ (used for the pseudo-orbit); 
if $\xi_{a_{i}}\in \Omega_{1}$, then we say that the rectangle has \emph{order} 0; if $\xi_{a_{i}}\in \Omega_{2}$ and $n^{+}(\xi_{a_{i}})=n>1$ or  $n^{-}(\xi_{a_{i}})=n>1$, then we say that the rectangle has  \emph{order}~$n$.

\begin{remark}
\label{rem-type00}
We observe that for each $n\ge 0$, there are only finitely many rectangles of order $n$; the   union of rectangles of order 0 covers $\Lambda\cap \Omega_{1}$ and the union of rectangles of orders $>1$ cover $\Lambda\cap\Omega_{2}$. 
The union of rectangles of order 0 can be included in a neighborhood of size $2\delta$ of~$\Omega_{1}$; the union of rectangles of higher order can be included into a neighborhood of $\Omega_{2}$ of size $2\delta$. 
\end{remark}

\subsubsection{The rectangles}
Remind that each rectangle of second generation $S_{a_{i}}\in \CS$ is obtained as a thinner rectangle of $T_{a_{i}}$. It thus inherits the order of $T_{a_{i}}$. 
As $\delta>0$ is taken small, none
of the rectangles of second generation and of order 0 which
intersects  $\Omega_3$ can intersect some rectangle of order
$n>1$. We say that a rectangle of order 0 that intersects
$\Omega_3$ is of order 00.

Therefore, each rectangle of order 00
intersects only with a finite number of other rectangles. Hence,
we can cut them as in
\cite{bowen}: let $S_{a_{i}}$ be a rectangle of order 00 and $S_{a_{j}}$ be
any other rectangle such that $S_{a_{i}}\cap S_{a_{j}}\not=\emptyset$. We set
\begin{eqnarray*}
S_{{a_{i}}{a_{j}}}^1&=&\{x\in S_{a_{i}}:\; W^u(x,S_{a_{i}})\cap S_{a_{j}}=\emptyset\quad\mbox{and}\quad
W^s(x,S_{a_{i}})\cap S_{a_{j}}=\emptyset\},\\
S_{{a_{i}}{a_{j}}}^2&=&\{x\in S_{a_{i}}:\; W^u(x,S_{a_{i}})\cap S_{a_{j}}\not=\emptyset\quad\mbox{and}\quad W^s(x,S_{a_{i}})\cap S_{a_{j}}=\emptyset\},\\
S_{{a_{i}}{a_{j}}}^3&=&\{x\in S_{a_{i}}:\; W^u(x,S_{a_{i}})\cap S_{a_{j}}\not=\emptyset\quad\mbox{and}\quad W^s(x,S_{a_{i}})\cap S_{a_{j}}\not=\emptyset\},\\
S_{{a_{i}}{a_{j}}}^4&=&\{x\in S_{a_{i}}:\; W^u(x,S_{a_{i}})\cap S_{a_{j}}=\emptyset\quad\mbox{and}\quad
W^s(x,S_{a_{i}})\cap S_{a_{j}}\not=\emptyset\}.
\end{eqnarray*}
Define $\CS_0$ as the set of points in $M$ belonging to some $S_{a_{i}}$ of order 00.  For $x\in\CS_0$ we set
$$\CR(x)=\{y\in M, \forall i,j\in \N, x\in
S_{{a_{i}}{a_{j}}}^k\Rightarrow y\in S_{{a_{i}}{a_{j}}}^k\}.$$
This defines a partition $\CR$ of $\CS_0$. By construction this partition is
finite and each of its elements  is stable under the map $[\,.\,,\,.\,]$. In other words, each element of $\CR$  is a rectangle.  These sets are
called \emph{rectangles of third generation}. If $x\in R_i\in \CR$, we set
$$W^u(x,R_i)= D^u_{\eps_{0}}(x)\cap R_i\quad \text{and}\quad
W^s(x,R_i)= D^s_{\eps_{0}}(x)\cap R_i.$$
 By construction,  if $x$ is a point in a rectangle of second generation $S_{a_{i_{1}}},\ldots S_{a_{i_{p}}}$, then
$$W^{u,s}(x,R_i)=W^{u,s}(x,S_{a_{i_{k}}})\cap R_i\quad\text{for $1\leq k\leq p$.}$$
Moreover, the 
next result follows as in \cite[Proposition 3.9]{leplaideur-aaa}, which means that the family $\CR$ is a Markov partition of $\CS_0$.
\begin{proposition}\label{prop-rec-mark-3}
Let $R_i$ and $R_j$ be  rectangles of third generation, $n\ge 1$ be
some integer and
$x\in R_i\cap f^{-n}(R_j)$. Then,
$$
f^n(W^u(x,R_i)) \supset  W^u(f^n(x),R_j)\quad\text{and}\quad
 f^n(W^s(x,R_i))\subset W^s(f^n(x),R_j).
$$
\end{proposition}

We also emphasize that, by construction, every point in a rectangle of third generation $\CR$ is a density point for $\CR$ and with respect to the unstable Lebesgue measure $Leb^{u}$.

\section{Proof of Theorem \ref{Theorem A}}\label{sec-proofthA}
In this section we prove the existence of a finite  SRB measure. In the first step we define hyperbolic times to be able to control de distortion of $\logju$ along orbits. These hyperbolic times define an induction into $\CS_{0}$ and we construct an invariant measure for this induction map in the second step of the proof. Here we also need to extend the construction to $\ol{\CS_{0}}$. In the last step, we prove that the return time is integrable and this allow to define the $f$-invariant  SRB measure. 
\subsection{Hyperbolic times}
Observe that every point in $ \CS_0$  is a  $\lambda$-hyperbolic
point and then, it  returns infinitely many often in $\Omega_3$.
Hence, every point in $\CS_0$ returns infinitely many often in
$\CS_0$.

Let  $x\in \CS_0$ and $n\in\N^*$ be such that $f^n(x)\in \CS_0$. Then, there exist two rectangles of third generation $R_l$
and $R_k$ such that $x \in R_l$ and $f^n(x)\in R_k$.
Thus, the Markov property given by Proposition~\ref{prop-rec-mark-3}
implies
$$f^{-n}(W^u(f^n(x),R_k))\subset W^u(x,R_l).$$
To achieve our goal we 
need some uniform distortion bound on
$$\prod_{k=0}^{n-1}\frac{J^u(f^k(x))}{J^u(f^k(y))},$$
where $y\in f^{-n}(W^u(f^n(x),R_k))$ and $J^u(z)$ is the
unstable Jacobian $\det df_{|E^u(z)}(z)$.
%
%
To obtain distortion bounds we will use the notion of hyperbolic
times introduced in~\cite{alves}.

\begin{definition}
Given $0<r<1$, we say that $n$ is a \emph{$r$-hyperbolic time}
for
$x$ if for every $1\leq k\leq n$
$$\prod_{i=n-k+1}^{n}\|df^{-1}_{|E^u(f^i(x))}\|\leq r^k.$$
\end{definition}

It follows from \cite[Corollary 3.2]{alves-bonatti-viana} that if
$$\liminf_\ninf
\frac1n\sum_{j=0}^{n-1}\log\|df^{-1}_{|E^u(f^j(x))}(f^j(x))\|<
2\log r,$$
then there exist infinitely  many $r$-hyperbolic times for
$x$.
Therefore, by construction of
$\Delta$, taking $r_0=e^{-\frac13\zeta}$ we have that for every $x$ in $\CS_0$,
there exist infinitely many
$r_0$-hyperbolic times for $x$.

\begin{remark}
\label{rem-densitehyperbotimes}
Another important consequence we shall use later is that there is a set with positive density of $r_{0}$-hyperbolic times. 
\end{remark}

\begin{lemma}\label{lem-hyper-hyperperto}
There exists $\delta'>0$ such that, if $\delta<\delta'$, then , for every $x$ in $\CS_0$, for every $r_0$-hyperbolic times for $x$, $n$, and for every  $y$ in $B_{n+1}(x,4\delta)$, the integer $n$ is also a $\sqrt r_0$-hyperbolic time for $y$.
\end{lemma}
\begin{proof}
Pick $\eps>0$  small compared to $\zeta$. 
By continuity of $df$ and $E^u$, there exists some $\delta'>0$ such that for all $x\in U$ and all $y\in B(x,4\delta')$, then
$$e^{-\eps}<\frac{\|df^{-1}_{|E^u(x)}(x)\|}{\|df^{-1}_{|E^u(y)}(y)\|}<e^{\eps}.$$
  Let us assume that $\delta<\delta'$. Take $x_0\in \CS_0$ and $n$ an $r_0$-hyperbolic time for $x$. Given $y\in B_{n+1}(x,4\delta)$, then for every $0\leq k\leq n$, we have that $f^{k}(y)\in B(f^{k}(x),4\delta)\subset B(f^{k}(x),4\delta')$, which means that for every $0\leq k\leq n$
\begin{equation}\label{equ-patch1.sept03}
e^{-\eps}<\frac{\|df^{-1}_{|E^u(f^k(x))}(f^k(x))\|}{\|df^{-1}_{|E^u(f^k(y))}(f^k(y))\|}<e^{\eps}.
\end{equation}
 The real number $\eps$ is very small compared to $\zeta$, thus (\ref{equ-patch1.sept03}) proves that for every $0\leq k\leq n$,
$$\|df^{k-n}_{E^u(f^{k}(y))}\|<(\sqrt{r_0})^{n-k}.$$
This proves that $n$ is a $\sqrt r_0$-hyperbolic time for every $y$.
\end{proof}

From here on we assume
that $\delta<\delta'$.
\begin{lemma}\label{lem-hyperretur-cs0}
If $x\in\CS_0$ and $n\ge 1$ is an $r_0$-hyperbolic time
for $x$, then $f^n(x)\in\CS_0$.
\end{lemma}
\begin{proof}
By definition of $r_0$-hyperbolic time,
$\disp\|df^{-1}_{|E^u(f^n(x))}\|<r_0=e^{-\frac13\zeta}$, which
implies that $f^n(x)\in \Omega_3$ (by definition of
$\eps_2$). Hence, $f^n(x)\in \CS_0$.
\end{proof}

Conversely, if $x$ and $f(x)$ belong to $\CS_0$, then 1 is a
$r_0$-hyperbolic time for $x$. Therefore, we say that a $r_0$-hyperbolic
time
for $x$ is
a {\it hyperbolic return  in $\CS_0$}.
Moreover, the Markov property of $\CR$ proves that, if
$n$ is a $r_0$-hyperbolic time for $x$, then there exists
$R_k\in \CR$ such that $f^n(x)$ is in $R_k$ and we have
$$f^{-n}(W^u(f^n(x),R_k))\subset \CS_0.$$
\subsection{Itinerary and cylinders}
For our purpose, we need to make precise what itinerary and cylinder mean. 
The rectangles of third generation have been built by intersection some rectangles of second generations (of type 00). Rectangles of second generation are restriction of rectangles of first generation. We remind that the set of rectangles of first generation satisfies an ``almost'' Markov property. 

Then, we say that two pieces of orbits $x,f(x),\ldots, f^{n}(x)$ and $y,f(y),\ldots, f^{n}(y)$ have the same itinerary if $x$ and $y$ have the same codes in $\S_{0}$ up to the time which represent $n$. As the dynamics in $\S_{0}$ is semi-conjugated to the one induced by $F$, this time may be different to $n$. 

The points in the third generation of rectangles having the same itinerary until some return time into $\CS_{0}$ define {\em cylinders}.


\subsection{SRB measure for the induced map}
For $x\in \CS_0$, we set  $\tau(x)=n$  if for some point $y$ in the same cylinder than $x$, $n$ is a $r_{0}$-hyperbolic time for $y$.
This allows us to define a map $g$ from $\CS_0$ to itself by $g(x)=
f^{\tau(x)}(x)$.

\begin{remark} As $\tau$ is not the first return time, the map $g$ is {\it a priori} not one-to-one. \end{remark}

As usual we set $$\tau^{1}(y)=\tau(y)\quad\text{and}\quad\tau^{n+1}(y)=\tau^{n}(y)+\tau(g^{n}(y)).$$ The $n$-cylinder  for $x$ will be the cylinder associated to $\tau^{n}(x)$.

Due to the construction of $\CS_{0}$, we have a finite set of  disjoint rectangles $R_{k}$,  a Markov map 
$$g:\cup R_{k}\to\cup R_{k},$$
Markov in the sense of Proposition \ref{prop-rec-mark-3}, and every point in every $R_{k}$ is a $\leb^{u}$ density point of $R_{k}$. Recall that $\CR(x)$ means $R_{k}$ if $x\in R_{k}$. 

The map $g$ can be extended to some points of $\ol{\CS_{0}}$. Note that $\ol{\CS_{0}}$ still have a structure of rectangle\footnote{A local sequence of (un)stable leaves $W^{u}(x_{n},R_{k})$ converges to a graph.}. More precisely, the extension of $g$ is defined for the closure of 1-cylinders: if $x$ belongs to $\CS_{0}$ and $g(x)=f^{n}(x)$, then we can define $g$ for all points $y$ in $f^{-n}\left(\ol{W^{u}(f^{n}(x),\CR(f^{n}(x)))}\right)\subset \ol{\CS_{0}}$ by 
$$g(y)=f^{n}(y).$$ 
By induction, $g^{n}$ is defined for the closure of $n$-cylinders. These points naturally belong to $\ol{\CS_{0}}$ but $\ol{\CS_{0}}$ is strictly bigger than the closure of $n$-cylinders.

We use the method in \cite{Young-Hu}, which  uses results from Section 6 in \cite{ledrap1}. Let $x_{0}$ be some point in~$\CS_{0}$ and set $\leb^{u}_{0}$ the restriction  and renormalization of $\leb^{u}_{x_{0}}$ to $\ol{W^{u}(x_{0},\CR(x_{0}))}$. We define
$$\mu_{n}= \frac1n\sum_{i=0}^{n-1}g^i_{*}(\leb^u_{0}).$$
We want to consider some accumulation point for $\mu_{n}$. The next lemma shows they are well-defined and $g$-invariant. 
\begin{lemma}
\label{lem-muweakstar-g}
There exists a natural way consider an accumulation point $\mu$ for $\mu_{n}$. It is a $g$-invariant measure.
\end{lemma}
\begin{proof}

To fix notations, let us assume that there are $N$ rectangles of third generation, $R_{1},\ldots,  R_{N}$. Each point in $\CS_{0}$ produces a code in $\{1,\ldots N\}^{\Z}$ just by considering its trajectory by  iterations of $g$\footnote{Actually the backward orbit of $g$ is not well defined. The projection is not  {\em  a priori} one-to-one, and to get a sequence indexed by $\Z$ we have to consider one inverse branch among the several possible pre-image by $g$.}. Nevertheless, it is not clear that every code  in  $\{1,\ldots N\}^{\Z}$ can be associated to a true-orbit in $M$. This can however be done for codes in the closure of the set of codes produces by $\CS_{0}$. This set has for projection the set of points which are in the closure of $n$-cylinders for every $n$. 

Now, note  that $g_{*}(\leb^{u}_{0})$ has support into the closure of the 1-cylinders. Then, 
the sequence of measures $\mu_{n}$ can be lifted in $\{1,\ldots, N\}^{\Z}$ and we can consider there some accumulation point for the weak* topology. Its support belongs to the set of points belonging to $n$-cylinders for every $n$, and the measure can thus be pushed forward in $M$. It is $g$-invariant. 
\end{proof}

\medskip
In the following we consider an accumulation point $\mu$ for $\mu_{n}$ as defined in Lemma~\ref{lem-muweakstar-g}.

%
%
%
%

We want to prove that $\mu$ is a  SRB measure. 
First, let us make precise what  SRB means for $\mu$. We remind that being  SRB means that the conditional measures are equivalent to the Lebesgue measure $\leb^{u}$ on the unstable leaves. This can be defined for any measure, not necessarily the $f$-invariant ones, and only requires that the partition into pieces of unstable manifolds is measurable (see \cite{rohlin1}).

\medskip
To
prove that $\mu$ is a  SRB measure, it is sufficient (and necessary) to prove
that there exists some constant $\chi$, such that for every integer
$n$, for every $y\in g^n(W^u(x_{0},\CR(x_{0})))$ and $y=f^{m}(x)$ for some $x\in W^{u}(x_{0},\CR(x_{0}))$, then
for every $z\in W^u(y,\CR(y))$
\begin{equation}
e^{-\chi}\leq \frac{\prod_{i=0}^{m-1}J^u(f^{-i}(y))}
{\prod_{i=0}^{m-1}J^u(f^{-i}(z))}
\leq e^{\chi}.
	\label{jaco-borne}
\end{equation}
First, recall that we have chosen the map $g$ in relation
with hyperbolic times. Hence we have some distortion bounds.

\begin{lemma}\label{lem-jaco-born-2}
There is $0<\omega_0<1$ such that for all $0\leq k\leq \tau(y)$,
$y\in \CS_0$ and $z\in f^{-\tau(y)}(W^u(g(y),\CR(g(y))))$ 
$$d^u(f^k(z),f^{k}(y))\leq\omega_0^{\tau(z)-k} d^u(g(z),g(y)).$$
\end{lemma}
\begin{proof}
If $y\in\CS_0\cap f^{-1}(\CS_0)$, then $g(y)=f(y)$ and
$y\in \Omega_3$ (far away from $S$).
If $g(y)=f^{\tau(y)}(y)$ with $\tau(y)>1$, then by definition
of $g$, there exists $y'$ such that 
\begin{enumerate}
\item $\tau(y)$ is a $r_0$-hyperbolic time for $y'$; and
\item $y'$ is in $\disp \CC(y)\pardef
f^{-\tau(y)}(W^u(f^{\tau(y)}(y),\CR(f^{\tau(y)}(y))))$.
\end{enumerate}
The diameter of $\CR(f^{\tau(y)}(y))$ is smaller than $2\beta(\delta)$. Hence, by construction of the third generation of rectangles, $\CC(y)$ is included into $B_{\tau(y)+1}(y,4\beta(\delta))$, which gives  that
 $\tau(y)$ is a
$\sqrt{r_0}$-hyperbolic time for every point in
$\CC(y)$. 
Therefore, the map
$$f^{k-\tau(y)}:W^u(f^{\tau(y)}(y),\CR(f^{\tau(y)}(y)))\rightarrow f^k(\CC(y))$$ is a contraction and it satisfies 
$$\|df^{k-\tau(y)}_{|E^u}\|\leq \sqrt{r_0}^{\tau(y)-k}.$$ Thus,
$$d^u(f^k(z),f^k(y))\leq r_0^{\frac{\tau(y)-k}{2}}d^u(g(z),g(y)).$$
\end{proof}

\begin{lemma}\label{lem-jaco-born-3}
There exist some constants $\chi_1>0$ and  $0<\omega<1$ such that
for every $n\geq 1$, for every $g^{n}(y)$ in $g^{n}(W^u(x,\CR(x)))$, for every $z$
in $f^{-\tau^{n}(y)}(W^u(g^n(y),\CR(y)))$ and for every $m\leq n$, we obtain
$$\sum_{j=0}^{\tau^m(y)-1}\Big|\log (J^u(f^j(z)))-\log(J^u(f^j(y)))
\Big|\leq \chi_1\,\omega^{n-m}.
$$
\end{lemma}
\begin{proof} The map $x\mapsto J^u(x)$ is H\"older-continuous because the map $E^u$ is H\"older-continuous; moreover it takes values in $[1,+\8[$.
The map $t\mapsto \log(t)$ is Lipschitz-continuous on  $]1,+\8[$. Thus, there exists some constants $\chi_2$ and $\alpha$, such that
\begin{equation}\label{equ2-patch0903}
\sum_{j=0}^{\tau^m(y)-1}\Big|\log (J^u(f^j(z)))-\log(J^u(f^j(y)))
\Big|\leq \chi_2\sum_{j=0}^{\tau^m(y)-1}(d^u(f^j(z),f^j(y)))^\alpha.
\end{equation}
 Hence, lemma \ref{lem-jaco-born-2} and (\ref{equ2-patch0903}) give
$$\sum_{j=0}^{\tau^m(y)-1}\Big|\log (J^u(f^j(z)))-\log(J^u(f^j(y)))
\Big|\leq \chi_2\left[\sum_{j=0}^{+\8}r_0^{j\alpha/2}\right](d^u(g^m(z),g^m(y)))^\alpha.$$Lemma \ref{lem-jaco-born-2} also yields $\disp d^u(g^m(z),g^m(y))\leq r_0^{(n-m)/2}\diam(\CR(g^{n}(y)))$.
\end{proof}

Lemma \ref{lem-jaco-born-3} gives that (\ref{jaco-borne}) holds for every $n$, for every $y\in g^n(W^u(x_{0},\CR(x_{0})))$ and for every $z\in W^u(y,\CR(y))$, with $\chi:=\chi_{1}+2\log\kappa$.

\begin{lemma}
\label{lem-mu-cs0}
Let $W^{s}_{loc}(\CS_{0})$ denote the set of local unstable leaf of points in $\CS_{0}$. 
The measure $\mu$ can be chosen so that $\mu(W^{s}_{loc}(\CS_{0}))=1$. 
\end{lemma}
\begin{proof}
Let us pick some rectangle of order 00, say $R_{k}$, having positive $\mu$ measure. Due to the Markov property, $g^{n}_{*}(\leb_{0})_{|R_{k}}$  is given by the unstable Lebesgue measure   supported on  a finite number of unstable leaves of the form $W^{u}(z, R_{k})$. 

Now recall  that stable and unstable leaves in $R_{k}$ are in $\CS_{0}$. 
Recall also that stable and unstable foliations are absolutely continuous. We can use the rectangle property of $R_{k}$ to project all the Lebesgue measures for $\supp g^{n}_{*}\leb_{0}\cap R_{k}$ on a fixed unstable leaf of $R_{k}$, say $\CF_{k}$. All these measure project themselves on a measure absolutely continuous with respect to Lebesgue, which yields that the projection on $\CF_{k}$ of $\mu_{n}$ can be written on the form 
$\varphi_{n,k} d\leb_{\ol{\CF_{k}}}.$
We have to consider the closure $\ol{\CF_{k}}$ to take account the fact that $\mu$ be ``escape'' from~$\CS_{0}$. 

It is a consequence of the bounded variations stated above (see \eqref{jaco-borne}) that all the $\varphi_{n,k}$ are uniformly (in $n$)  bounded $\leb_{\ol{\CF_{k}}}$ almost everywhere. 
In other words, all the $\varphi_{n,k}$ belong to a ball of fixed radius for $L^{\8}(\leb_{\ol{\CF_{k}}})$ and this ball is compact for the weak* topology. 

Therefore, we can consider a converging subsequence (for the weak* topology). As there are only finitely many $R_{k}$'s, we can also assume that the sequence converges for every rectangle $R_{k}$. For simplicity we write $\to_{\ninf}$ instead of along the final subsequence. 

Now, the convergence means that for every $\psi\in L^{1}(\leb(\ol{\CF_{k}}))$, 
$$\int\psi\varphi_{n,k}\,d\leb_{\ol{\CF_{k}}}\to_{\ninf}\int\psi\varphi_{\8,k}\,d\leb_{\ol{\CF_{k}}}.$$
If $\pi_{k}$ denote the projection on $\CF_{k}$
We can choose for $\psi$ the function $\BBone_{\ol{\CF_{k}}\setminus\CF_{k}}$. Then for every $n$, 
$$0\le \int\psi\varphi_{n,k}\,d\leb_{\ol{\CF_{k}}}=\int\psi\varphi_{n,k}\,d\leb_{\CF_{k}}\le\mu_{n}(\ol{R_{k}}\setminus R_{k})=0.$$ 
This shows that $\CF_{k}$ has full $\pi_{k*}\mu$ measure, and  as this holds for every $k$, this shows that $W^{s}_{loc}(\CS_{0})$ has full $\mu$ measure. 
\end{proof}

\begin{remark}
\label{rem-r0hyperreturn}
An important consequence  of Lemma \ref{lem-mu-cs0} is that for $\mu$ almost every point $x$, there exists a point $y\in \CS_{0}$ in its local stable leaf. Note that every $r_{0}$-hyperbolic time for $y$ is a return time for $y$, and then also for $x$. 
\end{remark}

%
%
%
%

\subsection{The SRB measure}
The map $g$ satisfies $g(z)=f^{\tau(z)}(z)$ for every $z$ in $\CS_{0}$.
Let $\CT(i)=\{z\in \CS_{0}\colon  \tau(z)=i\}$ be the set of points in $\CS_{0}$ such that
the first return time equals $i$. We set
$$\wh m=\sum_{i=1}^{+\8}\sum_{j=0}^{i-1}f^j_*(\mu|\CT(i)).$$
Then $\wh m$ is a $\sigma$-finite $f$-invariant measure and it is finite
if and only if $\tau$ is $\mu$-integrable; see~\cite{zweimuller}. In this case, we may
normalize the measure to obtain some probability measure. This will be an SRB measure. 

We define for $n\ge 1$ the set
$$H_n=\left\{x\in M\colon \text{$n$ is a $r_{0}$-hyperbolic time for $y$  in $W^{s}_{loc}(x)$}\right\}.$$
Observe that the following properties hold:
\begin{enumerate}
                \item[(a)] if $x \in H_j$ for $j\in\mathbb N$, then $f^i(x) \in H_m$ for any $1\le i<j$  and
$m=j-i;$
                \item[(b)] there is $\theta > 0$ such that for $\mu$ almost every $x$ in $\CS_{0}$
$$\displaystyle\limsup _{n\to\infty}\frac{1}{n}\#\{1\leq j\leq n:
x\in H_j\} \geq \theta ;$$
                \item[(c)]$H_n \subset  \{\tau \le n\}$ for each $n\ge 1$.
                                \end{enumerate}
                                
Condition (a) is a direct consequence of the definition for $H_{n}$. Condition (b) holds by construction of $\CS_{0}$, by  Pliss Lemma (see e.g. \cite{alves-bonatti-viana}) and the property emphasized in Remark \ref{rem-r0hyperreturn}.    Property (c) holds due to Lemma  \ref{lem-hyperretur-cs0}.


\begin{proposition}\label{pr.a-d}  The inducing time $\tau$ is $\mu$-integrable.
\end{proposition}

\begin{proof}
Suppose by contradiction that $\int \tau d\mu=\infty$.
By Birkhoff's Ergodic Theorem we have
$$\frac{1}{i}\sum_{k=0}^{i-1}\tau(g^{k}(x))\rightarrow \int \tau d\mu =\infty, $$
for $\mu$-almost every point $x\in \mathcal S_0$.  
Let $x \in \mathcal S_0$ be a $\mu$-generic point. Define, for every $i\in
\mathbb{N}$,
$$j_{i}=j_{i}(x)= \sum_{k=0}^{i-1}\tau(g^{k}(x)).$$ This means that
$g^{i}(x)= f^{j_{i}}(x)$. We define
$$\mathcal{I}=\mathcal{I}(x)=\{j_{1},j_{2},j_{3},...\}.$$
Given $j\in \mathbb{N}$, there exists a unique integer $r= r(j)\geq
0$ such that $j_r < j \leq j_{r+1}$. Supposing that $x\in H_{j}$,
then $g^{r}(x)\in H_{m}$, where $m=j-j_{r}$, by (a) above. 
%
%
%
%
%
For each $n$ we have
\begin{equation*}\label{integralretorno}
\frac{1}{n}\#\{j\leq n:x \in H_{j}\}\leq \frac{r(n)}{n}.
\end{equation*}
By construction, if $r(n)=i$, that is, $j_{i}<n \leq j_{i+1}$, then
$$
\frac{j_{i}}{i}< \frac{n}{r(n)}\leq
\frac{j_{i+1}}{i+1}\left(1+\frac{1}{i}\right).
$$
Since $$ \frac{j_{i}}{i}= \frac{1}{i}\sum_{k=0}^{i-1}\tau(
g^{k}(x))\rightarrow\infty,\text{ as } i\rightarrow +\infty,
$$
 it follows that
$$
\lim_{n\rightarrow\infty}\frac{1}{n}\#\left\{1\leq j\leq n : x \in
H_{j}\right\}= \lim_{n\to\infty}\frac{r(n)}{n}= 0.
$$
This contradicts item (b), and so one must have that the inducing
time function $\tau$ is $\mu$-integrable.
\end{proof}

\section{Proof of Theorem~\ref{Theorem B}}\label{sec-proofthB}
Let $\lambda>0$ be fixed. We consider a large positive constant $K\gg 2$ whose magnitude shall be adjusted at the end of the proof of Theorem\ref{Theorem B}. Remember that we defined above $$B(S,\eps_1)= \{y\in M:\, d(S,y)<\eps_1\},
\quad\text{and }\quad
 \Omega_0= \Omega\setminus B(S,\eps_1).$$
We pick $\eps_1>0$ (precise conditions will be stated along the
way) small enough so that in particular
$B(S,\eps_1)\subset U$.
By continuity, we can also choose $\eps_1>0$ small enough such that for
every
$x\in B(S,\eps_1)$, 
\begin{equation}\label{eq.epsun}
\disp
1\leq\|df(x)_{|E^u(x)}\|<e^{\frac{\lambda}{{K}}}\quad \text{and} \quad\disp
1\leq\|df^{-1}(x)_{|E^s(x)}\|<e^{\frac{\lambda}{{K}}}.
\end{equation}
Moreover, there exist $\lu>0$  and $\ls>0$
 such that for
every $x\in \Omega\setminus B(S,\eps_1)$ it holds:
\begin{enumerate}
\item for all $v\in E^u(x)\setminus\{0\}$
$$\|df(x)v\|_{f(x)}>e^{\lu} \|v\|_x\qand
\|df^{-1}(x)v\|_{f^{-1}(x)}<e^{-\lu} \|v\|_x;
$$
\item for all $v\in E^s(x)\setminus\{0\}$
$$\|df^{-1}(x)v\|_{f^{-1}(x)}>e^{\ls} \|v\|_x\qand
\|df(x)v\|_{f(x)}<e^{-\ls} \|v\|_x.$$
\end{enumerate}

\subsection{Graph transform}

\subsubsection{Constants to control lack of hyperbolicity}
We denote
by $|\,.\,|$ the Euclidean norm on $\R^N$, where $N=\dim M$. 
By continuity, we can assume that the maps $x\mapsto \dim
E^*(x)$ are constant, for $*=u,s$.
From now on, we will
denote by $\R^u$  the
space $\R^{\dim E^u}\times\{0\}^{\dim E^s}$.
In the same way $\R^s$,
 $B^u(0,\rho)$ and $B^s(0,\rho)$ will denote respectively the spaces $\{0\}^{\dim
E^u}\times\R^{\dim E^s}$, $B^{\dim
E^u}(0,\rho)\times\{0\}^{\dim E^s}$ and 
$\{0\}^{\dim E^u}\times B^{\dim E^s}(0,\rho)$.

\begin{proposition}\label{prop-lyapu-unif}
Let $\eps>0$ be small compared to $\lambda,$ $\lu$ or $\ls$. There are constants $\rho_1>0$, $0<K_1<K_2$, a positive function $\bar\rho$, and a family of embeddings
$\phi_x:B_x(0,\rho_1)\subset \R^N\rightarrow M$  with $\phi_x(0)=x$ such that\footnote{The notation $B_{x}$ is to remind that the ball defined in $R^{N}$ has its image centered at $x$.}
\begin{enumerate}
\item $d\phi_x(0)$ maps  $\R^u$ and $\R^s$ onto $E^u(x)$ and $E^s(x)$ respectively;
\item if $\wh f_x=\phi_{f(x)}^{-1}\circ f\circ\phi_x$ and $\wh f_x^{-1}=\phi_{f^{-1}(x)}^{-1}\circ f^{-1}\circ\phi_x$, then
\begin{enumerate}
\item  if $x\in \Omega\setminus B(S,\eps_1)$, then
\begin{enumerate}
\item for all $ v\in \R^u\setminus\{0\}$
$$|d\wh f_x(0).v|>e^{\lu} |v|\qand 
|d\wh f_x^{-1}(0).v|<e^{-\lu} |v|;$$
\item for all $v\in \R^s\setminus\{0\}$
$$
|d\wh f_x^{-1}(0).v|>e^{\ls} |v|\qand
|d\wh f_x(0).v|<e^{-\ls} |v|;
$$
\end{enumerate}
\item if $x\in B(S,\eps_1)$, then
\begin{enumerate}
\item for all $v\in \R^u\setminus\{0\}, $
$$
\quad\quad |v|\leq |d\wh f_x(0).v|<e^{\frac{\lambda}{{K}}} |v|\qand
|v|\geq |d\wh f_x^{-1}(0).v|>e^{-\frac{\lambda}{{K}}} |v|;
$$
\item for all $ v\in \R^s\setminus\{0\}, $
$$
\quad\quad |v|\leq |d\wh f_x^{-1}(0).v|<e^{\frac{\lambda}{{K}}} |v|\qand
|v|\geq |d\wh f_x(0).v|>e^{-\frac{\lambda}{{K}}} |v|;
$$
\end{enumerate}
\end{enumerate}
\item $0<\bar\rho(x)\leq \rho_1$ for every $x\in B(S,\eps_1)$ and
$\bar\rho(x)=\rho_1$ for every $x\in\Omega\setminus B(S,\eps_1)$;
\item on the ball $B_x(0,\bar\rho(x))$ we have $\lip (\wh f_x-d\wh
f_x(0))<\eps$ and $\lip (\wh f^{-1}_x-d\wh
f^{-1}_x(0))<\eps$;
\item for every $x$ and for every
$z,z'\in B_x(0,\rho_1)$,
$$K_1|z-z'|\leq d(\phi_x(z),\phi_x(z'))\leq K_2|z-z'|.$$
\end{enumerate}
\end{proposition}

This is a simple consequence of the map $\exp$ defined for every Riemannian manifold.
However, it is important for the rest of the paper to understand
that
the two constants $K_1$ and $K_2$ do not depend on $\eps_1$. They
result from the distortion due to the angle between the two sub-spaces
$E^u$ and $E^s$, plus the injectivity radius. These quantities are
uniformly bounded.

Moreover, as $\Omega$ is a compact set in
$U$, we can choose $\rho_1>0$ such
that
$B(\Omega,\rho_1)\subset U$. As the maps $x\mapsto E^u(x)$ and
$x\mapsto E^s(x)$ are continuous, we  can also assume that
$\rho_1$ is small enough so that for every $x\in\Omega$ and $y\in\phi_x(B(0,\rho_1))$, the slope of $d\phi_x^{-1}(y)(E^u(y))$ in
$\R^N=\R^u\oplus\R^s$ is smaller than $1/2$.

\subsubsection{General Graph transform}
Here, we recall some well known basic statements for the graph transform. 
Let $E$ be some Banach space and $T:E\rightarrow E$ be a linear
map such that there exists a $T$-invariant splitting $E=E_1\oplus
E_2$. We set $T_i = T_{|E_i}$ and we assume that the norm on $E$
is adapted to the splitting, i.e.
$\|.\|_E=\max(\|.\|_{E_1},\|.\|_{E_2})$. We also assume that there
exist  $\lambda_2<0<\lambda_1$ such that:
\begin{enumerate}
\item $\|T_1^{-1}v\|_{E_1}\leq
e^{-\lambda_1}\|v\|_{E_1}$,\,  for every $v$ in $E_1$;
\item  $\|T_2v\|_{E_2}\leq
e^{\lambda_2}\|v\|_{E_2}$,\,  for every $v$ in $E_2$.
\end{enumerate}

\begin{proposition}\label{prop-transfo-graph}
Let $\rho$ be a positive number, $\eps>0$ be small
compared to $\lambda_1$ and $-\lambda_2$ and $F:E\to E$ be a $C^1$ map
 such that  $F(0)=0$,
 $\lip(F-T)<\eps$ and $\lip(F^{-1}-T^{-1})<\eps$ on the ball $B(0,\rho)$.
Then
\begin{enumerate}
\item the image by $F$ of the graph of any map
$g:B_1(0,\rho)\rightarrow B_2(0,\rho)$ satisfying $g(0)=0$ is a graph
of some map $\Gamma(g):B_1(0,\rho e^{\lambda_1-2\eps})\rightarrow
B_2(0,\rho e^{-\lambda_2+2\eps})$;
\item $F$ induces some operator $\Gamma$ on the set $\lip_1$ of
$1$-Lipschitz continuous maps $g:B_1(0,\rho)\rightarrow B_2(0,\rho)$
with $g(0)=0$;
\item  $\Gamma$ is a contraction on $\lip_1$ (with the standard norm on the space of
Lipschitz continuous maps), and thus it
admits some unique fixed point.
\end{enumerate}
\end{proposition}
For the proof of Theorem~\ref{Theorem B}, it is important to keep in mind the two key
points in the proof of Proposition~\ref{prop-transfo-graph}. On the one
hand, the fact
that $\Gamma$ is a contraction on $\lip_{1}$ is essentially due to the
spectral gap of $dF$. This is obtained by the properties of the
$T_{i}$'s and Lipschitz proximity of $F$ and $T$.
On the other hand, the fact that the image by $F$ of
any graph (from $B_{1}(0,\rho)$ to $B_{2}(0,\rho)$) extends beyond the boundary of the ball
$B(0,\rho)$ is essentially due to expansion on $E_{1}$.

\subsubsection{Adaptations to our case}
We want to use Proposition~\ref{prop-transfo-graph} in the
fibered case of $\wh f_x$.  If
$x$ is in $\Omega_0\cap f^{-1}(\Omega_0)$ the spectral gap of
$d\wh f_{x}(0)$ is uniformly bounded
from below in $B_{x}(0,\rho_{1})$, and so we can apply the result
with $d\wh f_{x}(0)$, $\wh f_{x}$ and $\rho=\rho_1$.  In $B(S,\eps_1)$, the value of
$\bar\rho(x)$ has to decrease to 0 when $x$ tends to $S$: the
spectral gap of $df$ tends to 0 as $x$ tends to $S$ because $k^u(x)+k^s(x)$ tends to 0. The idea
is to apply Proposition~\ref{prop-transfo-graph} for $d\wh f_{x}^n(0)$
 and $\wh f_{x}^n$ for some good $n$. First, we have to check that the
 assumptions of the proposition hold.

\begin{proposition}\label{prop-prA-kergo} 
Let $x\in \Omega_0$ and $n\geq 2$ be such that $f^n(x)\in \Omega_0$ and
 $f^k(x)\in B(S,\eps_1)$ for all $1\le k\le n-1$. 
There exists some constant $C>0$ such that for every $0<r\leq 1$,
Proposition \ref{prop-transfo-graph} holds for $F=\wh f^n_{x}$,
$T=dF(0)$
and $\rho= Cre^{-\frac{9n\lambda}{{2K}}}$, and also for  $F=\wh f^{-n}_{f^n(x)}$, $T=d \wh f^{-n}_{f^n(x)}(0)$ and $\rho=Cre^{-\frac{9n\lambda}{{2K}}}$.

\end{proposition}
\begin{proof}
We refer the reader to Proposition 2.3 in \cite{leplaideur-aaa}. We emphasize that that proof used $d^{2}f$ to actually get Lipschitz continuity for $df$. Now, this is a standard computation to exchange Lipschitz estimates with H\"older estimates. 
\end{proof}

\begin{remark}
\label{rem-9K}
We checked that the constant 9 that appeared in Proposition 2.3 in \cite{leplaideur-aaa} also works if we only use a H\"older continuity condition on $df$. We point out that these constants are very common in the Pesin theory and have to be understood ``in spirit'': we take very small fraction of the global minimal expansion ratio. 
\end{remark}

Let $x\in\Omega_0$ and $C>0$ be given by Proposition~\ref{prop-prA-kergo}. We remind that $n^{\pm}(x)$ were defined in Definition \ref{def-n+}. 
 If $f(x)\in B(S,\eps_1)$ and $n^+(x)<+\8$, then we
set
\begin{equation}\label{eq.elf}
l^f(x)=Ce^{-\frac{9 \lambda}{{2K}}n^+(x) }.
\end{equation}
Analogously, if $f^{-1}(x)\in B(S,\eps_1)$ and $n^-(x)<+\8$, then
we set 
\begin{equation}\label{eq.elb}
l^b(x)= Ce^{-\frac{9\lambda}{{2K}}n^-(x)}.
\end{equation}
If $x\in \Omega_{0}\cap
f^{-1}(\Omega_{0})$,
the graph transform due to Proposition~\ref{prop-transfo-graph} with
$F=\wh f_{x}$, $T=d\wh f_{x}(0)$ and $\rho=\rho_{1}$ will be called the \emph{one-step graph transform}. If $x\in\Omega_{0}\cap f^{-1}(B(S,\eps_{1}))$, the graph transform
due to Proposition~\ref{prop-prA-kergo} with  $F=\wh f^{n^+(x)}_{x}$,
$T=dF(0)$
and $\rho= Cre^{-\frac{9\lambda}{{2K}}n^+(x)}$will be called
the \textit{$n^{+}(x)$-steps graph transform}. Both will be denoted
by~$\Gamma_{x}$.

\subsection{Truncations}

We are
going to show that if $x\in\Omega_0$ is a $\lambda$-hyperbolic point of bounded type, then we are able to construct a piece of unstable leaf as some set
$\phi_x(\graph(g_x))$, where $g_x$ will be some special map from
$B^u_{x}(0,l(x))$ to $B^s_{x}(0,l(x))$ satisfying $g_x(0)=0$ for
some $l(x)\leq \rho_{1}$. We will use the graph transform along the
backward orbit of $x$.

The next result shows that every
$\lambda$-hyperbolic point must return infinitely often to~$\Omega_{0}$, both in the future and in the past.

\begin{lemma}\label{lem-prA-n+-fini}
Let $\xi\in \Omega_0$ be some $\lambda$-hyperbolic point such that
$f(\xi)\notin\Omega_0$ (resp. $f^{-1}(\xi)\notin\Omega_0$). Then
$n^+(\xi)<+\8$ (resp. $n^-(\xi)<+\8$).
\end{lemma}

\begin{proof}
Assume that $f(\xi)\notin\Omega_0$ and $n^+(\xi)=+\8$. This means that
$ f^n(\xi)\in B(S,\eps_1)$ for all $n\ge1$.
This implies that for all $n\geq 1$ and all $v\in E^s(f^n(\xi))$
$$
\|df(f^n(\xi))v\|\geq e^{-\frac{n\lambda}{{K}}}\|v\|.
$$
Hence we get
$$\liminf_\ninf\frac1n\log\|df^n_{|E^s}(\xi)\|\geq
-\frac{\lambda}{{K}},$$
which contradicts the fact that $\xi$ is $\lambda$-hyperbolic.
The other case is proved similarly.
\end{proof}

Let $x\in\Omega_0$ be a $\lambda$-hyperbolic point. Given $n\ge0$, we define $x_{n}= f^{-n}(x)$.
By Lemma~\ref{lem-prA-n+-fini} there exist integers $0\leq q_{0}<p_{0}\leq
q_{1}<p_{1}\leq \ldots$ such that for all $i\ge 1$
\begin{itemize}
	\item  $x_{k}\in \Omega_{0}$ for all $0\leq k\leq q_{0}$;
\item   $x_{k}\in
B(S,\eps_{1})$ for all  $q_{i}<k<p_{i}$;

\item  $x_{k}\in
	\Omega_{0}$ for all $p_{i-1}\leq k\leq q_{i}$.
\end{itemize}
For $i\ge0$ we set
\begin{equation}\label{eq-mi}
y_{i}= x_{q_{i}}=f^{-q_{i}}(x),\quad z_{i}= x_{p_{i}}=f^{-p_{i}}(x)\quad\text{and}\quad m_{i}=p_{i}-q_{i}.
\end{equation}
  We define $\Gamma^n_{x}$ as the composition of the
graph transforms along the piece of orbit
$x_{n},\ldots,x_{1}$, where we take the one-step graph
transform $\Gamma_{x_{k}}$ if $x_{k}$ and $x_{k-1}$ lie in $\Omega_{0}$,
and the $m_{i}$-steps graph transform if $x_{k}$ is one of the
$z_{i}$'s.

Our goal  is to prove that the sequence of maps $\Gamma^n_{x}(\wh
0_{x_{n}})$
converges to some map, where~$\wh 0_{x}$ denotes the null-map from
$B^u_{x}(0,\rho_{1})$ to $B^s_{x}(0,\rho_{1})$.  This is well known for uniformly
hyperbolic dynamical systems, but here the critical set $S$
influences the graph transform:
by construction, $\Gamma^n_{x}$ is a contraction of ratio smaller
than $$\exp\left(-\left(n-\sum_{i, p_{i}\leq n}(m_{i}-1)\right)\left(\lu+\ls-4\eps\right)\right).$$
But it is not clear that the length of the graph associated to
$\disp \Gamma^n_{x}(\wh 0_{x_{n}})$ is uniformly (in $n$) bounded away from 0.  
 
 Let us now explain how $S$ influences the graph transform and specially the length of the graphs associated to the
$\Gamma^n_{x}(\wh 0_{x_{n}})$'s.
Given  $x_{n}\in \Omega_{0}$, there exists an integer $k$ such that $p_{k}\leq n\leq q_{k+1}$.
For $p_{k}+1\leq i\leq n$ we just apply the one-step graph transform,
and so we get in $B_{z_{k}}(0,\rho_{1})$ some graph. At this moment we
apply the $m_{k}$-steps graph transform to the graph restricted to the
ball $B_{z_{k}}(0,l^f(z_{k}))$, and we get some graph in $B_{y_{k}}(0,l^b(y_{k}))$.  We call this phenomenon a
\emph{truncation}.

\begin{figure}
\includegraphics[scale=0.6]{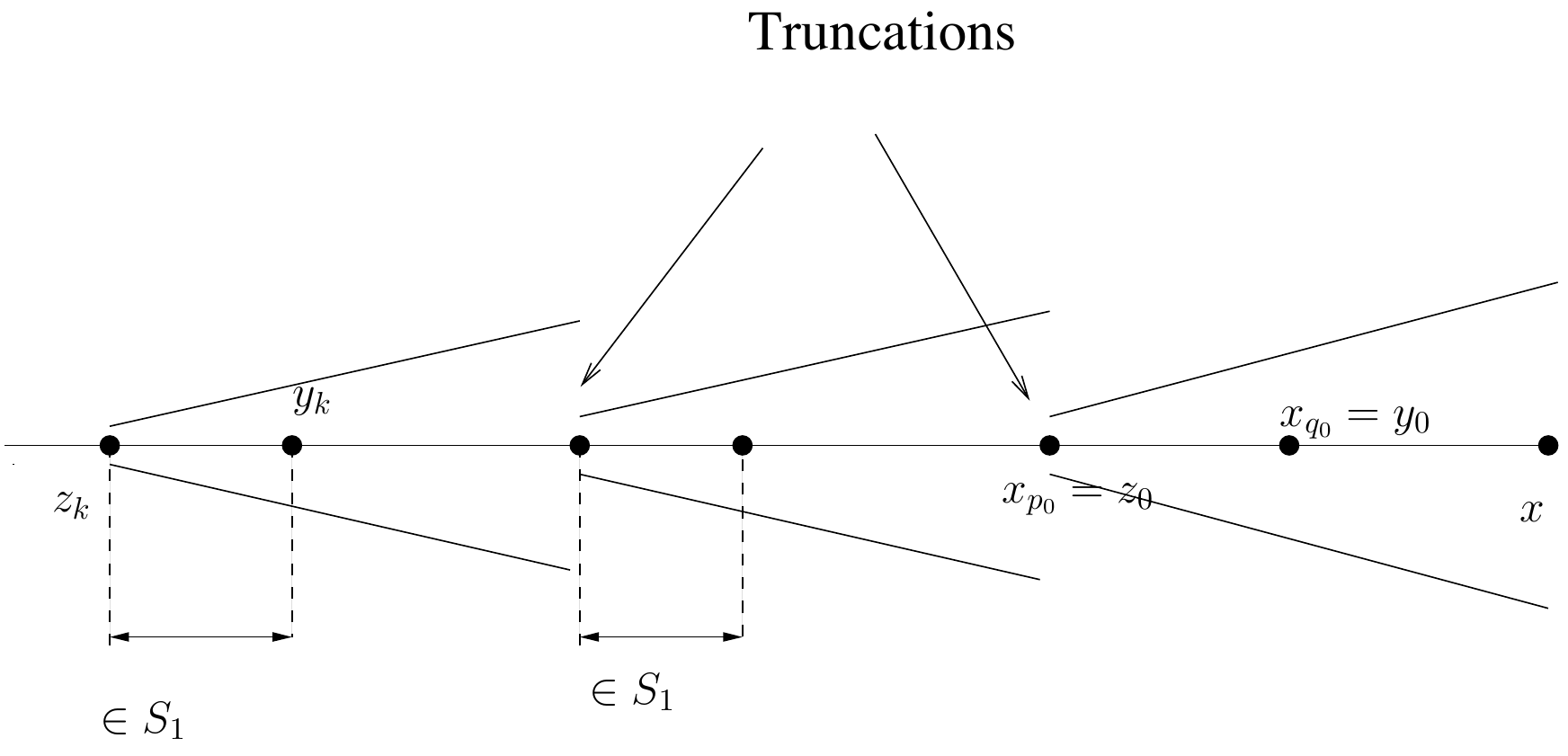}
\caption{Truncations due to excursions close to $S_{1}$. The vertical represents the size of the graphs.}
\label{fig-truncation}
\end{figure}

For $p_{k-1}\leq i\leq q_{k}$, we can again apply the one-step graph
transform, but to the small piece of graph with length $l^b(y_{k})$. However, the length increases along this piece of orbit
(as long as it is smaller than $2\rho_{1}$) because, at each (one-)step,
we
can take the whole part of the image-graph. For $i=p_{k-1}$ three cases may occur:
\begin{itemize}
	\item[(i)]  the length of the graph is $2\rho_{1}$ and so, there is
	a new truncation;

	\item [(ii)]  the length of the graph is strictly smaller than
	$2\rho_{1}$, but is bigger than $2l^f(z_{k-1})$: again, there is a new truncation;

	\item [(iii)]  the length of the graph, $2l$, is strictly smaller than $2l^f(z_{k-1})$: 
	we apply the $m_{k-1}$-steps graph transform with $\rho = l$, and we obtain some graph in $B_{y_{k-1}}(0,l)$.
	From~$x$, we can see the truncation due to $\Gamma_{z_{k}}$.
\end{itemize}
We proceed this way along the piece of orbit $y_{k-1},\ldots,x_{-1}$.
Hence, the length of the graph $\Gamma^n_{x}(\wh
0_{x_{n}})$ is (at least) equal to $2l^f(z_{j})$, where $j$ is such that $p_{j}$ is
the smallest integer where a truncation occurs. We set $i(n):=j$. 

The main problem we have to deal with, is to know if  the sequence $i(n)$ may go to $+\8$ or if it is bounded. In the first case, this means that the sequence of graphs $\Gamma^n_{x}(\wh
0_{x_{n}})$ converges to a graph with zero length  in any ball $B(x,\rho)$. On the contrary, the later case produces a limit graph with positive length. In Proposition~\ref{prop-key-paststab} below we show that, for $\lambda$-hyperbolic points of bounded type,  it is possible to have the sequence 
$(i(n))_n$ bounded. For its proof we need two preliminary lemmas.

Let $A=\max_{x\in M}{\|df(x)\|}$ and define
$$
\gamma=\left({\log A-\frac{\lambda}2}\right)\left({\frac{\lambda}2-\frac\lambda{K}}\right)^{-1}.
$$
We emphasize that $\gamma$ converges to some finite positive number when $K$  grows to infinity.

\begin{lemma}
\label{lem-nique-nuh-ze1}
Given a $\lambda$-hyperbolic $x$, let  $m$ be an integer such that $f^{j}(x)\in B(S,\eps_{1})$ for all $0\le j\le m-1$. 

\begin{enumerate}
\item If  $m>\gamma q$ for some integer $q\ge 1$, then
 $$\|df^{m+q}_{|E^u(x)}\|\le e^{(m+q){\lambda}/2}.$$
\item If $\|df^{q}_{|E^u(f^{m}(x))}\|\le e^{q\lambda/2}$, then for every $0\le j\le m-1$
$$||df^{q+j}(x)||\le e^{(q+j)\lambda/2}.$$
\end{enumerate}

\end{lemma}

\begin{proof}
Let $\w$ be a vector in $E^{u}(x)$. Using that $f^{j}(x)\in B(S,\eps_{1})$ for $0\le j\le m-1$ and also 
 \eqref{eq.epsun}, we deduce
 \begin{eqnarray*}
\|df^{m+q}(x).\w \|&=& \|df^{q}(f^{m}(x))df^{m}(x).\w\|\\
&\le & A^{q}e^{m\frac\lambda{K}}\|\w\|\\
&\le & e^{q\log A+m\frac\lambda{K}}\|\w\|.
\end{eqnarray*}
Now, using that $m>\gamma q$, we obtain
$$
m\left(\frac{\lambda}2-\frac\lambda{K}\right)        > q\left(\log A-\frac{\lambda}2\right)
$$
and so
$$
q\log A+m\frac\lambda{K} < (m+q)\frac{\lambda}2.
$$
This concludes the proof of the first item.

Let us now prove the second item. 
We set $x_{l}:=f^{l}(x)$. 
 $df(x_{j})$ expands every vector in the unstable direction for $0<j\le m$ by a factor lower than $e^{\frac\lambda{K}}<e^{\frac\lambda2}$.  Hence, the expansion in the unstable direction for $df^{q+j}(x)$ is strictly smaller than 
$$e^{q\frac\lambda2+j\frac\lambda{K}}<e^{+(q+j)\frac\lambda2}.$$

\end{proof}

\begin{lemma}
\label{lem-nique-nuh-ze2}
Given a $\lambda$-hyperbolic point $x$, let $y_k$, $q_{k}$, $p_{k}$ and $m_{k}$ be as in \eqref{eq-mi}. Given $0<\eps<1$,  let $0\le s\le q_k$ be the largest  integer such that $\log\|df^{-s}(x)|_{E^{u}(x)}\|\le -s\lambda(1-\eps)$. If 
$$\log\|df^{q_{k}}(y_{k})|_{E^{u}(y_{k})}\|<{\frac{9m_{k}\lambda}{2K}},$$
then 
$$
m_{k}>\frac{2K}{9} s(1-\eps).
$$
\end{lemma}

\begin{proof} If $s=0$, then nothing has to be proved. Assume now that $s\ge 1$. As we are assuming  that  $\log\|df^{q_{k}}(y_{k})_{|E^{u}(y_{k})}\|<\frac{9m_{k}\lambda}{2K}$ and $s\le q_{k}$, then 
$$
\log\|df^{s}(x_s)_{|E^{u}(x_s)}\|<\frac{9m_{k}\lambda}{2K},
$$
because $df$ always expands in the unstable direction. 
Also, by the choice of $s$, we have $\log\|df^{-s}(x)|_{E^{u}(x)}\|\le -s\lambda(1-\eps)$, 
and so 
$$s\lambda(1-\eps)<\frac{9m_{k}\lambda}{2K}.$$
\end{proof}

%
%

\subsection{Proof of Theorem \ref{Theorem B}}
We can now state the key proposition to construct the local unstable manifolds.
\begin{proposition}
\label{prop-key-paststab}
There is a constant $K>0$ such that if $x$ is a $\lambda$-hyperbolic of bounded type, then the sequence of truncations $(i(n))_n$ is bounded. 
\end{proposition}
\begin{proof}
We consider an increasing sequence of integers $(s_{n})_n$ such that $$\disp\lim_{\ninf}\frac1{s_{n}}\log\|df^{-s_{n}}(x)|_{E^{u}(x)}\|\le-\lambda.$$ 
We also consider $0<\varepsilon<\frac13$ and $N$ sufficiently big such that for every $n\ge N$, 
$$\frac1{s_{n}}\log\|df^{-s_{n}}(x)|_{E^{u}(x)}\|<-\lambda(1-\eps)\quad
\text{and}\quad \frac{s_{n+1}}{s_{n}}<L,$$
for some real number $L$ (which exits by definition of $\lambda$-hyperbolic  of bounded type).

\bigskip
\noindent
Let us pick $K$. 
Several conditions on $K$ are stated along the way. First we assume that $K$ is sufficiently big such that $\disp\frac1K<1-\eps$ holds.

\bigskip\noindent
Consider $k$ such that $q_{k}\ge s_{N}$ and let
$0\le s\le q_k$ be as in Lemma~\ref{lem-nique-nuh-ze2}. Note that we necessarily have $s\ge s_{N}$. 
Then, consider the biggest $n$ such that $s_{n}\le s$. Note that  $n\ge N$, as $(s_n)_n$ is an increasing sequence.

Assume now, by contradiction, that 
\begin{equation}
\label{equ-hypo-absu}
\log\|df^{q_{k}}(y_{k})|_{E^{u}(y_k)}\|<\frac{9m_{k}\lambda}{2K}.
\end{equation}
In such case, by Lemma \ref{lem-nique-nuh-ze2} we have 
\begin{equation}
\label{equ-mkK}
\disp m_{k}>\frac{2K}{9} s(1-\eps).
\end{equation}
Now we distinguish the two possible cases:

\begin{enumerate}
\item $s<q_{k}$.
\smallskip

\noindent 
We claim that, in this case, $s_{n+1}$ is bigger than $p_{k}$. Indeed, by definition of $s$ and $n$, we cannot have $s_n<s_{n+1}\le q_k$ . 
 Therefore, $s_{n+1}$ is bigger than $p_{k}$ and we get

$$\frac{s_{n+1}}{s_{n}}>\frac{q_{k}+m_{k}}{s}>1+\frac{2K}{9}(1-\eps)$$
which can be made bigger than $L$ if $K$ is chosen sufficiently big. 

\item $s=q_{k}$.
\smallskip

\noindent It may happen that the contraction of $df^{-q_{k}}(x)$ in the unstable direction is so strong that $s_{n+1}=q_{k}+1$. Actually, this property may endure for several iterates. In other words, there may be several $s_{j}$ between $q_{k}$ and $p_{k}$. Nevertheless, we can show that there will always be a (too) big gap between some of them.

Consider $K$ sufficiently big so that 
$$\disp \frac{2K}9(1-\eps)>\gamma^{2}+2\gamma>\gamma.$$ 
Using  \eqref{equ-mkK} and also that we are considering the case $s=q_k$, we obtain $m_{k}>\gamma q_{k}$. 
 Then, we apply Lemma~\ref{lem-nique-nuh-ze1} with $m=m_{k}$ and $q=q_{k}$, and this shows that $(\gamma+1)q_{k}$ is not one of the $s_{j}$. 
Note that
$$m_{k}>\frac{2K}9(1-\eps)q_{k}>(\gamma^{2}+2\gamma)q_{k},$$
holds if $K$ is sufficiently big (remember that $\gamma$ is bounded). 
This yields $m_{k}-\gamma.q_{k}>\gamma(1+\gamma)q_{k}$. Now, we are back to the case (1) with $(\gamma+1)q_{k}$ instead of $q_{k}$ and $m_{k}-\gamma.q_{k}$ instead of $m_{k}$. 
 Thus,  there are no $s_{j}$'s between $(\gamma+1)q_{k}$ and $q_{k}+m_{k}$.  Observe that this last interval has length bigger than 
$$
\left(\frac{2K}{9}(1-\eps)-1-\gamma\right)q_{k}.
$$
If $s_{j}$ is the biggest term of the sequence $(s_{l})$ smaller than $(\gamma+1)q_{k}$, then $s_{j+1}>p_{k}=m_{k}+q_{k}$ and  we get
$$
\frac{s_{j+1}}{s_{j}}\ge \frac{p_{k}}{(1+\gamma) q_{k}}=\frac{m_{k}+q_{k}}{(1+\gamma)q_{k}}
\ge 
\left(\frac{2K(1-\eps)}{9}+1\right)\left({1+\gamma}\right)^{-1},
$$
which can again be made bigger than $L$ by a convenient choice of $K>0$. 

\end{enumerate}

Both cases yield a contradiction, which means that the assumption \eqref{equ-hypo-absu} does not hold. 
This shows that for every $k$ such that $q_{k}>s_{N}$, 
$$i(p_{k})\le \max\{i(n), n\le s_{N}\}.$$
As truncation only occurs for integers of the form $p_{k}$, the proposition is proved. 
\end{proof}

With the previous notations, the length $l_{n}$ of the graph associated
to
$\Gamma^n_{x}(\wh 0_{x_{n}})$ is smaller than $2\rho_{1}$  and is actually fixed by the finite number of truncations $\disp \max\{i(n), n\le s_{N}\}$, where $N$ appears in the proof of Proposition \ref{prop-key-paststab}. 
This proves that the
sequences of lengths $(l_{n})$ is bounded away from 0. The family of
maps $\Gamma^n_{x}(\wh 0_{x_{n}})$ converges to some map
$$g_{x}:B^u_{x}(0,l(x))\rightarrow B^s_{x}(0,l(x)),$$ with
$g_{x}(0)=0$, where $l(x)$ is such that $0<l(x)\leq l_{n}$ for every $n$.
We define
 $$\CF^u_{\loc}(x)=
\phi_x(\graph(g_x)).$$
Furthermore, Lemma \ref{lem-prA-n+-fini} proves that the backward orbit of
$x$ returns infinitely often to~$\Omega_0$. Therefore, if
$f^{-k}(x)$ belongs to $B(S,\eps_{1})$ we set
$$\CF^u_{\loc}(f^{-k}(x))= f^n(\CF^u_{\loc}(f^{-(n+k)}(x))),$$
where $n$ is the smallest positive integer such that $f^{-(n+k)}(x)$
belongs to $\Omega_{0}$.
Then, we set
$$\CF^u(x)=\bigcup_{n,\ f^{-n}(x)\in
\Omega_0}f^n(\CF^u_{\loc}(f^{-n}(x))).$$
The uniqueness of the map $g_x$ and its construction prove that
$\CF^u(x)$ is an
immersed manifold.

To complete the proof of Theorem~\ref{Theorem B}, we must check that
the $\CC^{1}$-disks that have been constructed are tangent to the
correct spaces.
Let $y$ be in $\CF^u_{\loc}(x)$. By construction of $\CF^u_{\loc}(x)$ we
have for every $n$ such that $f^{-n}(x)\in
\Omega_0$, $f^{-n}(y)\in \CF^u_{\loc}(f^{-n}(x))$. For such an
integer $n$, we pick
some map
 $g_{n,y}:
B^u(0,\rho_1)\rightarrow B^s(0,\rho_1)$ such that $f^{-n}(y)\in
\phi_{f^{-n}(x)}(\graph(g_{n,y}))$ and
$T_{f^{-n}(y)}\phi_{f^{-n}(x)}(\graph(g_{n,y}))=E^u(f^{-n}(y))$.
As the
map $g_x$ is obtained as some unique fixed-point for the graph
transform, the sequence $\Gamma^n_x(g_{n,y})$ converges to
$g_x$. By
$df$-invariance of $E^u$ (until the orbit of $y$ leaves $U$) we must
have
$T_y\CF^u(x)=E^u(y)$.

We also can do the
same construction with $f^{-1}$ to obtain some immersed manifolds
$\CF^s(x)$. Then, $x\in \CF^u(x)\cap \CF^s(x)$. It might be important to have an estimate for
the length of $\CF^u_{\loc}(x)$ or $\CF^u(x)$.  However, we are not able
to
give a lower bound for such estimates.



\begin{thebibliography}{10}

\bibitem{alves-bonatti-viana}
J.~{Alves}, C.~{Bonatti} and M.~{Viana}.
\newblock SRB measures for partially hyperbolic systems whose central direction
  is mostly expanding.
\newblock {\em Inventiones Math.}, 140:351--298, 2000.

\bibitem{alves}
J.~F. Alves.
\newblock S{RB} measures for non-hyperbolic systems with multidimensional
  expansion.
\newblock {\em Ann. Sci. \'Ecole Norm. Sup. (4)}, 33(1):1--32, 2000.

\bibitem{young-benedicks}
M.~{Benedicks} and L.-S. {Young}.
\newblock Sinai-Bowen-Ruelle measures for certain H\'enon maps.
\newblock {\em Invent. Math.}, 112(3):541--576, 1993.

\bibitem{Viana-Bonatti}
C.~{Bonatti} and M.~{Viana}.
\newblock {SRB} measures for partially hyperbolic systems whose central
  direction is mostly contracting.
\newblock {\em Israel Journal of Math.}, 1999.

\bibitem{bowen}
R.~{B}owen.
\newblock {\em {E}quilibrium {S}tates and the {E}rgodic {T}heory of {A}nosov
  {D}iffeomorphisms}, volume 470 of {\em {L}ecture notes in {M}ath.}
\newblock {S}pringer-{V}erlag, 1975.

\bibitem{Herman}
M.~{H}erman.
\newblock Construction de diff\'eomorphismes ergodiques.
\newblock notes non publi\'ees.

\bibitem{Hu}
H.~{Hu}.
\newblock Conditions for the existence of sbr measures for {`Almost Anosov
  Diffeomorphisms'}.
\newblock {\em Transaction of AMS}, 1999.

\bibitem{Young-Hu}
H.~{Hu} and L.-S. {Young}.
\newblock Nonexistence of {{\it SBR}} measures for some diffeomorphisms that
  are {"Almost Anosov"}.
\newblock {\em Ergod. Th. \& Dynam. Sys.}, 15, 1995.

\bibitem{ledrap1}
F.~{L}edrappier and L.-S. {Y}oung.
\newblock {T}he metric entropy of diffeomorphisms {P}art {I}:
  {C}haracterization of measures satisfying {P}esin's entropy formula.
\newblock {\em {A}nnals of {M}athematics}, 122:509--539, 1985.

\bibitem{leplaideur-aaa}
R.~Leplaideur.
\newblock Existence of {SRB} measures for some topologically hyperbolic
  diffeomorphisms.
\newblock {\em Ergodic Theory Dynam. Systems}, 24(4):1199--1225, 2004.

\bibitem{rohlin1}
V.~{R}ohlin.
\newblock {O}n the fundamental ideas of measure theory.
\newblock {\em {A.M.S.T}ranslation}, 10:1--52, 1962.

\bibitem{zweimuller}
R. Zweim\"uller:   
\newblock {I}nvariant measures for general(ized) induced transformations.
\newblock {\em Proc. Amer. Math. Soc.} 133 (2005), 2283--2295. 



\end{thebibliography}

\end{document}